\documentclass[12pt]{article}
\usepackage[utf8]{inputenc}
\usepackage{geometry}
\usepackage{amsmath, amsfonts, amssymb, amsthm}
\usepackage{commath}
\usepackage{mathtools}
\usepackage[table,xcdraw]{xcolor}
\usepackage[inline]{enumitem}
\usepackage{tikz}
\usepackage{romannum}
\usepackage{graphicx}
\usepackage{subcaption}
\usepackage{caption}
\usepackage{sidecap}
\usepackage{comment}
\usepackage{parskip}
\usepackage[hidelinks]{hyperref}
\usepackage{url} 
\usepackage{doi}
\let\oldforall\forall
\let\forall\undefined
\DeclareMathOperator{\forall}{\oldforall}

\newtheorem{theorem}{Theorem}[section]
\newtheorem{lemma}[theorem]{Lemma}
\newtheorem{corollary}[theorem]{Corollary}
\newtheorem{prop}[theorem]{Proposition}
\theoremstyle{remark}
\newtheorem{remark}[theorem]{Remark}
\theoremstyle{definition}
\newtheorem{definition}[theorem]{Definition}
\newtheorem*{definition*}{Definition}

\newcommand{\quotes}[1]{``#1''}

\newcommand{\Z}{\mathbb{Z}}
\newcommand{\R}{\mathbb{R}}
\newcommand{\N}{\mathbb{N}}
\newcommand{\SL}{\mathrm{SL}}
\newcommand{\dist}{\mathrm{dist}}

\newcommand{\BA}{\mathbf{Bad}}
\newcommand{\rr}{\mathbf{r}}
\renewcommand{\aa}{\mathbf{a}}
\newcommand{\bb}{\mathbf{b}}

\newcommand{\bgamma}{\boldsymbol{\gamma}}
\newcommand{\xx}{\mathbf{x}}
\newcommand{\pp}{\mathbf{p}}

\newcommand{\yy}{\mathbf{y}}

\renewcommand{\epsilon}{\varepsilon}

\title{\textbf{Hausdorff dimension of differences of badly approximable sets}}
\date{\vspace{-1em}}

\author{Dorsa Hatefi, David Simmons}

\begin{document}
\pagenumbering{arabic}
\maketitle

\begin{abstract}
 The set of badly approximable numbers, $\BA$, is known to be winning for Schmidt's game and hence has full Hausdorff dimension. It is also known that the set of inhomogeneously badly approximable numbers has full dimension. We prove that the set difference also has full dimension using a variant of the Schmidt game, which we call the \textit{rapid game}, played on the space of unimodular grids. 

\end{abstract}
\section{Statement of the problem and main results}
From the theory of Diophantine approximation, a real number $\alpha$ is called \textit{badly approximable}, if there exists a constant $c$ (depending on $\alpha$) such that \[\abs{\alpha-\frac{p}{q}}>\frac{c}{q^2}\] for all $p \in \Z$ and $q \in \Z_{\neq 0}$.
Let us denote the set of all such numbers as $\BA$.  $\BA$ is an example of a set with Lebesgue measure zero and full Hausdorff dimension, equal to the dimension of the whole real line. If unfamiliar with the definitions, the reader is referred to the notes~\cite{aspectsofrecentwork}.
\par A natural generalisation of this set is the set of \textit{inhomogeneous} badly approximable numbers. Informally, one can imagine \quotes{shifting} all the rational points on the real line by some parameter $\gamma$ in the numerator and approximating the real number $\alpha$ with these shifted rationals. More precisely, \[\BA^{\gamma}=\{\alpha \in \R : \exists \ c(\alpha,\gamma) \text{ s.t. } \abs{\alpha-\frac{p+\gamma}{q}}>\frac{c}{q^2} \, \forall \ (p,q) \in \Z \times \Z_{\neq 0} \}.\]

This set also has Lebesgue measure zero and full Hausdorff dimension~\cite{threeproblems,Einsiedler2011}. The full dimension proof in~\cite{threeproblems} uses the known homogeneous interval construction and with a \quotes{+ ($\gamma - \gamma = 0$) technique} gets the inhomogeneous statement. This means it proves the stronger statement that $\BA^{\gamma} \cap \BA$ has full dimension. Our aim would be to look at the set $\BA^{\gamma} \setminus \BA$, motivated by attempting to study $\BA^{\gamma}$ independently of $\BA$.

In what follows, without loss of generality, we restrict $\alpha$ to the interval $[0,1)$.
Note that sets above can be interpreted in terms of unit circle rotations. $\BA$ can be viewed as the set of angles $\alpha$ for which as $q$ runs through the integers, the circle rotations $q \alpha$ stay away from $0$ (the starting point), while $\BA^{\gamma}$ would be the set for which the rotations quantitatively avoid some other irrational point on the circle.
More precisely, our set of interest $\BA^{\gamma} \setminus \BA$ can be interpreted to be the set of points $\alpha \in [0,1)$ such that their $q\alpha$ (modulo $1$) orbits avoid the shrinking target $(\gamma-\frac{c}{q},\gamma+\frac{c}{q})$ for some $c>0$, but for every $\epsilon>0$ visits the interval $(-\frac{\epsilon}{q},\frac{\epsilon}{q})$ around $0$ infinitely often. 
\par
A stronger notion than having full Hausdorff dimension is to be a \textit{winning} set, in the sense of the classical \textit{Schmidt's game} first introduced in~\cite{Schmidt1966}, which we recall in Section~\ref{sec:3}. $\BA$ is a winning set~\cite{Schmidt1966}, and in fact it has been shown that $\BA^{\gamma}$ is also winning~\cite{Einsiedler2011}. Winning sets have the important intersection property (see Section~\ref{sec:3}) implying $\BA^{\gamma} \cap \BA$ is also winning. Now considering our set of interest $\BA^{\gamma} \setminus \BA$, this set cannot be winning in the sense of the classical Schmidt game since it is disjoint from another winning set $\BA^{\gamma}$ (see Section~\ref{sec:3}). Therefore, instead, we introduce a variant of the Schmidt game that does not have the intersection property but for which \quotes{winning} still implies full dimension.  A consequence of the new game is the following statement.

\begin{theorem} \label{maintheorem}
For all $\gamma \in \R \setminus \Z$, the set $\BA^{\gamma} \setminus \BA$ has full Hausdorff dimension.
\end{theorem}
The proof of Theorem \ref{maintheorem}. consists of working on $\BA^{\gamma} \setminus \BA$ from a dynamical point of view, in particular, using flows on the space of unimodular grids. In Section~\ref{sec:2}, our dynamical language is described. In Section~\ref{sec:3}, we introduce our variant of Schmidt's game and its dynamical reformulation and strategy.

\begin{remark}
    It is possible to generalise this problem to higher dimensions. However, to extend the ideas in this paper, other techniques such as parametric geometry of numbers may be necessary, see e.g. \cite{DFSU}. It would also be interesting to study the reverse set difference, i.e.\ $\BA \setminus \BA^{\gamma}$, or more generally, to explore under what conditions on the pair $(\gamma , \delta)$, the set $\BA^{\gamma} \setminus \BA^{\delta}$ has full dimension. We leave these for a future work. 
\end{remark}

\textbf{Acknowledgements.} D.H. would like to thank Lifan Guan for useful discussions. D.H. was supported by a Departmental Studentship from the Department of Mathematics, University of York.
D.S. was supported by a Royal Society University Research Fellowship, URF\textbackslash R1\textbackslash 180649.

\section{Dynamical interpretation of Theorem \ref{maintheorem}}
\label{sec:2}

In this section, we will work with a geometric and dynamical interpretation of Diophantine approximations. The unfamiliar reader is referred to for example the lecture notes~\cite{kleinbock2010metric}.

Let us introduce a slightly different notation to define $\BA$ and $\BA^{\gamma}$ which will help our intuition in what follows.
For $\epsilon > 0 $, define
\[\BA_\epsilon=\left\{ \alpha \in \R : q |p-q\alpha | > \epsilon \   \forall \ (p,q) \in \Z \times \N  \right\}.\]
Therefore we have
\[\BA=\bigcup_{\epsilon>0}\BA_\epsilon.\]
Similarly, we can define the set of inhomogeneous badly approximable numbers as follows:
\[\BA^{\gamma}=\bigcup_{\epsilon>0}\BA^{\gamma}_{\epsilon},\]
where
\[\BA^{\gamma}_{\epsilon}=\left\{ \alpha \in \R : q | p-q\alpha+\gamma | > \epsilon, \forall \  (p,q) \in \Z \times \N \right\}.\]

Due to the Dani correspondence principle~\cite{Dani1985}, many Diophantine sets have a natural dynamical reformulation. We first briefly recall the homogeneous setting. Let $X_2$ be the space of unimodular lattices in $\R^2$. Then we have the isomorphism 
\[X_2 \cong \SL_2(\R)/\SL_2(\Z)\]
with the smooth left action of $\SL_2(\R)$ 
\[\SL_2(\R) \times \SL_2(\R)/\SL_2(\Z) \rightarrow \SL_2(\R)/\SL_2(\Z) \quad (h,g\SL_2(Z)) \mapsto hg\SL_2(\Z).\]
\[SL_2(\R) \times X_2 \rightarrow X_2 \quad (g,\Lambda) \mapsto g\Lambda .\]

Now for any $x \in \R$, let
\[u_x= \begin{pmatrix}
    1 & -x \\
    0 & 1
\end{pmatrix} \quad \text{ and } \quad \Lambda_x=u_x\Z^2=\{(p-qx,q): p,q\in \Z\} \in X_2 .\]Note that in this manner, we have attributed to each real number a unimodular lattice from the space $X_2$. Now for any $t \in \R$ let
\[g_t = \begin{pmatrix}
    e^t & 0\\
    0 & e^{-t}
\end{pmatrix}.\]
Then we have the following:
\begin{theorem}[Dani correspondence] For  $x \in \R$ we have 
\[x \in \BA \iff \text{The orbit } \{g_t\Lambda_x:t>0\} \text{ is bounded} .\]
\end{theorem}
For $\epsilon > 0$, define the set $K_\epsilon$ as follows:
\[K_\epsilon = \{\Lambda \text{ unimodular lattice} : \Lambda \cap B(0, \epsilon) = \{0\}\}.\]
Then, by Mahler's compactness criterion~\cite[Theorem~11.33]{EinsiedlerWard}, $x \in \BA$ if and only if there exists $\epsilon>0$ such that the trajectory $g_t\Lambda_x$ stays within $K_\epsilon.$ In the following, we introduce the set $F_\epsilon$ as the corresponding set in the inhomogeneous case.\\
\subsection{Space of unimodular grids}
In the inhomogeneous setting, rather than unimodular lattices, the objects of interest are \textit{unimodular grids}. By a unimodular grid, we simply mean any translation of a unimodular lattice, hence a grid can be written as $\Lambda + \bgamma$ where we call $\Lambda \in X_2$ the \textit{homogeneous} part and $\bgamma \in \R^2$ the \textit{translation vector}. Two grids are identified in the following way:
\[\Lambda + \bgamma = \Lambda' + \bgamma' \iff \Lambda=\Lambda' \text{ and } \bgamma'-\bgamma \in \Lambda.\]

Therefore, the space of unimodular grids $Y_2$ has the following identification
\[Y_2 \cong \SL_2(\R) \ltimes \R^2 / SL_2(\Z) \ltimes \Z^2\] 
with the smooth action
\[SL_2(\R) \times Y_2 \rightarrow Y_2 \quad (g,(\Lambda+\bgamma)) \mapsto g\Lambda + g\bgamma .\]

\par Let $\widetilde X_2$ be the image of $X_2$ under the natural embedding of $X_2$ in $Y_2$ such that $\Lambda \mapsto \Lambda + \mathbf{0}$ and define the natural projection map $\pi:Y_2 \rightarrow X_2, \enspace \Lambda+\bgamma \mapsto \Lambda$ which recovers the homogeneous part. 
Similar to $K_\epsilon$, we define the set $F_\epsilon$ within the space of unimodular grids as follows:
\[F_{\epsilon} = \{\Lambda \text{ unimodular grid} : \Lambda \cap B(0, \epsilon) = \emptyset \}.\]
Note that instead of the intersection being $\{0\}$ (which is a point in every lattice) we require it to be empty. Therefore, $F_\epsilon$ is the set of grids (with any homogeneous part) that avoid an $\epsilon$-neighbourhood of the origin, informally, avoid becoming a lattice.
Now we are ready to state the inhomogeneous Dani correspondence~\cite{Einsiedler2011}.
Let $\Lambda_{x,\gamma}=\Lambda_x+\widetilde{\gamma}\in Y_{2}$, where $\widetilde{\gamma}=(\gamma, 0)$.

\begin{theorem}[Inhomogeneous Dani correspondence]
\label{theoremdani}
For  $x \in \R$ and $\gamma \in \R \setminus \Z$, we have
\[
x\in \BA^{\gamma} \iff  \text{The orbit }\{g_t \Lambda_{x, \gamma}: t>0\} \text{ remains in $F_\epsilon$  for some $\epsilon > 0$.}
\]
 \end{theorem}

 With the above interpretation, Theorem \ref{maintheorem} can be rephrased as follows:
 \begin{theorem}
 \label{maintheoremv2}
    {The set of all x for which the orbit $g_t \Lambda_{x, \gamma}$ stays within $F_\epsilon$ for some $\epsilon>0$ while escaping $\pi^{-1}(K_\epsilon)$ for every $\epsilon>0$ has full Hausdorff dimension.}
 \end{theorem}

\section{A variant of Schmidt's game}
\label{sec:3}
\subsection{Review of the classical game}
First we briefly remind ourselves of the classical Schmidt game played on a complete metric space, in our case simply $\R$. The game is between two players \textbf{Alice} and \textbf{Bob}, and depends on two fixed variables $0<\alpha,\beta<1$.
\textbf{Bob} starts the game by choosing a closed ball $B_0$ of radius $\rho_0$, followed by \textbf{Alice} who chooses a closed ball $A_0$ inside $B_0$ of radius $\alpha\rho_0$. In the next round, \textbf{Bob} chooses a closed ball $B_1$ inside $A_0$ of radius $\alpha \beta \rho_0$ followed by \textbf{Alice}.
In a similar manner, each player at their turn chooses a ball inside the other player's last ball with \textbf{Alice} multiplying the radius by $\alpha$, and \textbf{Bob} multiplying the radius by $\beta$. \\
This generates a nested sequence of closed balls with a unique point of intersection $x_\infty$:
\[\{x_\infty\}=\bigcap_{n=0}^\infty A_n= \bigcap_{n=0}^\infty B_n,\]
which we call the $\textit{outcome}$. A set $S$ is called \textit{$(\alpha,\beta)$-winning} if \textbf{Alice} can play in a way that $x_\infty$ lies in $S$ regardless of how \textbf{Bob} plays. The set $S$ is called \textit{$\alpha$-winning} if it is \textit{$(\alpha,\beta)$-winning} for all $\beta \in (0,1)$. Finally, we call a set \textit{winning}, if it is \textit{$\alpha$-winning} for some $\alpha$.
\par The following are some of the key consequences of winning proved in~\cite{Schmidt1966}:
\begin{itemize}
    \item winning sets have full Hausdorff dimension;
    \item the intersection of countably many winning sets is winning. 
\end{itemize}

It is proved that $\BA$ and $\BA^{\gamma}$ are winning and so it immediately follows that both are of  full dimension and moreover  that 
$$
\dim_{H}(\BA^{\gamma} \cap \BA)= 1  \, . 
$$ 
\begin{remark}\label{intersectionpropertyremark}
    Note that two disjoint sets cannot both be winning, since otherwise the intersection property would lead to a contradiction. Therefore it follows that $\BA^{\gamma} \setminus \BA$ is  not a winning set. 
\end{remark}

\subsection{The strong game}
We will also utilise the strong game introduced by McMullen \cite{McMullen_absolute_winning}. The rules of the strong game are the same as the rules of the classical Schmidt game, except that instead of \textbf{Alice} and \textbf{Bob} being required to choose balls with radii exactly $\alpha\rho_n$ and $\alpha\beta\rho_n$ respectively, \textbf{Alice} chooses a ball of radius $\rho_n' \geq \alpha\rho_n$ and \textbf{Bob} chooses a ball of radius $\rho_{n+1} \geq \beta \rho_n'$. The rules for determining the outcome of the game $x_\infty$ are the same; a set $S$ is called \emph{$(\alpha,\beta)$-strong winning} if \textbf{Alice} can play in a way such that $x_\infty$ lies in $S$ regardless of how \textbf{Bob} plays, and the notions of \emph{$\alpha$-strong winning} and \emph{strong winning} are defined analogously.

\begin{theorem}[\cite{McMullen_absolute_winning}]
Strong winning sets are winning.
\end{theorem}

\subsection{The rapid game}
We introduce the following variant of Schmidt's game. Let $0<\alpha,\beta<1$ as above and consider a new variable $\alpha_n\in (0,\alpha]$ which \textbf{Bob} has to choose at his $n$-th turn, such that rather than $\alpha$, \textbf{Alice} should multiply the radius by $\alpha_n$.
More precisely:
\begin{itemize}
    \item At the $0$-th stage, \textbf{Bob} chooses a ball $B_0$ of radius $\rho_0$, together with a parameter $\alpha_0 \in (0,\alpha]$.
    \item At the $n$-th stage, \textbf{Bob} chooses a ball $B_n=B(x_n,\rho_n)$ together with a parameter $\alpha_n \in (0,\alpha]$. Then \textbf{Alice} chooses a ball $A_n=B(y_n,\alpha_n \rho_n) \subseteq B_n$.
    \item At the $(n+1)$-th stage, \textbf{Bob} chooses a ball $B_{n+1}=B(x_{n+1},\rho_{n+1}) \subseteq A_n$ with $\rho_{n + 1} = \beta \alpha_n \rho_n$, and a parameter $\alpha_{n+1} \in (0,\alpha]$. 
\end{itemize}
If  $\inf_n \alpha_n >0$, we say \textbf{Alice} wins by default. Otherwise, we say \textbf{Alice} wins if and only if the unique intersection point $x_\infty$ of the sequence $B_n$ (i.e. the outcome) lies in the target set $S$. We call the set $S$ \textit{$(\alpha,\beta)$-rapid winning} if \textbf{Alice} has a strategy where she wins regardless of how \textbf{Bob} plays. A set is defined to be \emph{$\alpha$-rapid winning} or \emph{rapid winning} analogously to the concepts of $\alpha$-winning and winning. Note that the key difference between the rapid game and Schmidt's game is that \textbf{Bob} has to eventually choose very small values of $\alpha_n$, to avoid losing by default.\\
\begin{prop}
    Rapid-winning implies full Hausdorff dimension. 
\end{prop}
\begin{proof} 
    Take the following choice of the sequence $\alpha_i$ for \textbf{Bob}:
    \begin{equation}  \label{seq} 
    \frac{\alpha}{2},\alpha,\frac{\alpha}{4},\alpha,\alpha,\frac{\alpha}{8},\alpha,\alpha,\alpha,\frac{\alpha}{16},\dots
    \end{equation}
    so that he will not lose by default. Note that this sequence satisfies
    \[\frac{\alpha}{2} \leq \sqrt[n]{\alpha_1 \dots \alpha_n}.\]
    Let $m=\lfloor \frac{1}{\beta} \rfloor$ be the number of pairwise disjoint intervals of length $\beta r$ fitting in an interval of length $r$. Fix a winning strategy for \textbf{Alice}, and define a map $\pi:\{1,\ldots,m\}^\N\to S$ as follows: $\pi(\omega)$ is the result of \textbf{Alice} playing her winning strategy, \textbf{Bob} choosing the sequence $(\alpha_i)$ according to \eqref{seq}, and on the $n$th turn, choosing $m$ disjoint intervals of size $\beta\alpha_n\rho_n$ in \textbf{Alice}'s previous move $A_n$, and choosing the $\omega_n$th of these intervals for his next play. Let $S^* = \pi(\{1,\ldots,m\}^\N) \subseteq S$, and note that $S^*$ is compact. 

Let $\{B_i\}$ a collection of intervals covering $S^*$ with
radii $\rho_i$. Since $S^*$ is compact we can without loss of generality assume $\{B_i\}$ is finite. Let
\[\alpha_1 \dots \alpha_{k_i+1}\beta^{k_i+1}\rho_0 \leq \rho_i < \alpha_1 \dots \alpha_{k_i} \beta^{k_i} \rho_0.\]
Choose $J$ large enough such that $\rho_i \geq \alpha_1 \dots \alpha_J\beta^J$ for all $i$. Then counting
the number of game intervals covered
at stage $J$ we get: 
\[
m^{J} \leq \sum_i 2m^{J-k_i}
\]
This implies
\begin{equation*}
    \begin{split}
        \frac{1}{2} \leq \sum_{i}m^{-k_i}
        & = \sum_i ((\sqrt[k_i]{\alpha_1 \dots \alpha_{k_i}}\beta)^{\frac{\log m}{\log \sqrt{\alpha_1 \dots \alpha_{k_i}}\beta}})^{-k_i}\\
        & = \sum_i ({\alpha_1 \dots \alpha_{k_i}\beta^{k_i}})^{\frac{\log m}{|\log {\sqrt{\alpha_1 \dots \alpha_{k_i}}}\beta|}}\\
        & \leq \sum_{i} (\rho_i)^{\frac{\log m}{|\log {\sqrt{\alpha_1 \dots \alpha_{k_i}}}\beta|}}\\
        & \leq \sum_{i} (\rho_i)^{\frac{\log m}{|\log\frac{\alpha}{2}\beta|}}
    \end{split}
\end{equation*}
Taking the infimum over all such covers $\{B_i\}$ gives
\[
\mathcal H^{\frac{\log m}{|\log \frac\alpha 2 \beta|}}(S) \geq \frac 12\cdot
\]
Letting $\beta \to 0$ we get $\dim_H(S)=1$.
\end{proof}

The above Proposition, together with the following Theorem, naturally implies Theorem~\ref{maintheorem}.

\begin{theorem}
    $\BA^{\gamma} \setminus \BA$ is \textit{rapid-winning}.
\end{theorem}


\begin{lemma}
If $S$ is $(\alpha,\beta)$-strong winning, $\alpha' \leq \alpha$, and $\beta \leq \tfrac 12\alpha\beta'$, then $S$ is $(\alpha',\beta')$-rapid winning.
\end{lemma}
\begin{proof}
Suppose \textbf{Alice} has a winning strategy for the $(\alpha,\beta)$-strong game. Consider a move $B_n = B(x_n,\rho_n)$ for \textbf{Bob} in the $(\alpha',\beta')$-rapid game, together with a choice of parameter $\alpha_n\leq\alpha'$. We will treat such a move as also being a move for \textbf{Bob} in the strong game. \textbf{Alice} responds in the strong game with a ball $A_{n,1} = B(x_{n,1},\rho_{n,1}) \subseteq B_n$ with $\rho_{n,1}\geq \alpha \rho_n$. \textbf{Bob} responds by choosing the ball $B_{n,1} = B(x_{n,1},\tfrac 12 \rho_{n,1})$. \textbf{Alice} responds with $A_{n,2} \subseteq B_{n,1}$, and so on until $\tfrac 12\alpha\rho_{n,k} \leq \alpha_n \rho_n$. Then \textbf{Alice} responds in the rapid game with the move $A_{n+1} = B(x_{n,k},\alpha_n\rho_n)$. We then denote \textbf{Bob}'s response in the rapid game as $B_{n+1} = B(x_{n+1},\rho_{n+1})$ with $\rho_{n+1} = \beta' \alpha_n\rho_n$.

We need to show that $A_{n+1}\subseteq B_{n,k}$, and that $B_{n+1}$ is a legal move for \textbf{Bob} in the strong game after \textbf{Alice} plays $A_{n,k}$. For the former, note that since $\alpha_n \leq \alpha' \leq \alpha$, we have $\rho_{n,1} \geq \alpha_n \rho_n$, and if $k > 1$ we have and thus since $\rho_{n,k}/\rho_{n,k-1} \geq \tfrac 12 \alpha$ and $\tfrac 12\alpha\rho_{n,k-1} > \alpha_n \rho_n$, which implies $\rho_{n,k} \geq \alpha_n \rho_n$. Either way, we get $\rho_{n,k} \geq \alpha_n\rho_n$ which implies $A_{n+1}\subseteq B_{n,k}$.

For the latter, it suffices to show that $\rho_{n+1} \geq \beta \rho_{n,k}$. But $\tfrac 12\alpha\beta'\rho_{n,k} \leq \beta'\alpha_n \rho_n = \rho_{n+1}$ so it suffices to take $\beta \leq \tfrac 12\alpha\beta'$.
\end{proof}

\begin{corollary}
\label{corollarySWimpliesRW}
Any strong winning set is rapid winning.
\end{corollary}

\subsection{Reformulating the games on grids} \label{sec:3.3}
Here, we reformulate the Schmidt, strong, and rapid games on the space of unimodular grids $Y_2$ instead of $\R$. At each stage, instead of a ball $B(x,\rho)$, the game state will be of the form $g_t u_x \Lambda_{*}$. Here, $\Lambda_{*} = \Z^2 + \Tilde{\gamma}$ is the initial state, and $u_x$ and $g_t$ encode the centre and radius of the ball respectively, such that $x$ is the centre and $t=-\tfrac12 \log(\rho)$ where $\rho$ is the radius. For example, the rapid game would be reformulated as follows:
\begin{itemize}
    \item At the $0$-th stage, \textbf{Bob} applies a $u_{x_0}$ transformation to $\Lambda_{*}$ followed by a $g_{t_0}$ transformation, with $t_0=-\tfrac12 \log(\rho_0)$ getting the grid $\Lambda_{B_0}=g_{t_0}u_{x_0}\Lambda_*$. He then chooses a parameter $s_0$ with $s_0=-\tfrac12 \log(\alpha_0)$, so $s_0 \in [-\tfrac12 \log\alpha , +\infty )$.  
    \item At the $n$-th stage, \textbf{Bob} chooses a grid  $\Lambda_{B_n}=g_{t_n}u_{x_n}\Lambda_*$, together with a parameter $s_n\in [-\log\alpha , +\infty )$. Then \textbf{Alice} chooses a grid of the form $g_{t_n'} u_{x_n'} \Lambda_{*}$, and from the above reformulation, $e^{-2t_n'}$ will be her radius and $x_n'$ will be her centre. Therefore we get $t_n'=t_n+s_n$, and $x_n'=y_n$.\\
    We can also describe her move in terms of application of transformations to \textbf{Bob}'s grid:
    \begin{equation*}
        \begin{split}
        g_{t_n+s_n}u_{y_n}\Lambda_* & = g_{s_n}g_{t_n}u_{x_n+e^{-2t_n}z_n}\Lambda_* \\
        & = g_{s_n}g_{t_n}u_{e^{-2t_n}z_n}u_{x_n}\Lambda_*\\
        & = g_{s_n}u_{z_n}g_{t_n}u_{x_n}\Lambda_* \\
        & = g_{s_n}u_{z_n}\Lambda_{B_n},
        \end{split}
    \end{equation*}
    where in the first line we let $y_n=(y_n-x_n)+x_n$ and define $z_n=\frac{y_n-x_n}{e^{-2t_n}}$ to be the distance of the centres rescaled according to the radius, so $\|z_n\| \leq 1-e^{-2s_n}$.
    \item At the $(n+1)$-th stage, with the same argument, \textbf{Bob} chooses the grid:
    \[\Lambda_{B_{n + 1}}= g_{-\log \beta} u_{w_n} \Lambda_{A, n}= g_{t_n+s_n-\log \beta} u_{x_n+ e^{-2t_n} (z_n+e^{-2s_n} w_n)} \Lambda_* ,\]
    where $w_n=\frac{x_{n+1}-y_n}{e^{-2(s_n+t_n)}}$ is the distance of the centres properly rescaled.
\end{itemize}
As explained above, at each step instead of choosing a ball, a $u_x$ transformation followed by a $g_t$ transformation is applied to the grid, where \textbf{Alice} chooses $z_n$ and \textbf{Bob} chooses $w_n$ and $s_n$. Note that since the radii tend to zero, after an infinite number of plays, the trajectory of the game state will lie within finite Hausdorff
distance of a trajectory $(g_tu_{x_\infty}\Lambda_*)_{t \geq 0}$. As before, \textbf{Alice} wins if $x_\infty$ lies in the target set.

\section{Proof of the main result}
\subsection{The intuitive idea}
We first attempt to describe the strategy in a more intuitive manner and then proceed to the formal proof.
\par 
Using Dani's correspondence we need to describe a strategy which ensures that for some $\epsilon>0$ the game state stays in the region 
\[F_{\epsilon} = \{\Lambda+\rr \text{ unimodular grid} : (\Lambda+\rr) \cap B(0, \epsilon) = \emptyset \},\]
but leaves $\pi^{-1}(K_\epsilon)$ for every $\epsilon>0$.
Recall from Section~\ref{sec:2}  that $K_\epsilon$ is the region
\[K_\epsilon = \{\Lambda \text{ unimodular lattice} : \Lambda \cap B(0, \epsilon) = \{0\}\},\]
and $\pi$ is just the projection map from the space of unimodular grids to lattices.
\par
Since $\BA^{\gamma}\cap \BA$ is strong winning, by Corollary~\ref{corollarySWimpliesRW} it is also rapid winning. This suggests that there is a rapid winning strategy which for any initial game state $\Lambda + \rr$ brings the game trajectory in $F_\epsilon \cap \pi^{-1}(K_\epsilon)$ for some $\epsilon > 0$ independent of $\Lambda+\rr$ (see Section~\ref{default} for a proof of this). Call this strategy the \textit{default strategy}. \textbf{Alice}'s strategy will be to use the default strategy for most of the time, but occasionally use a different strategy which we call the \textit{auxiliary strategy} (elaborated in Section~\ref{auxiliary}) that ensures the game state leaves $ \pi^{-1}(K_\epsilon)$ for every $\epsilon > 0$.
\par
 Designing \textbf{Alice}'s auxiliary strategy involves choosing a parameter $x$ for her $u_x$ transformation which is followed by a $g_t$ transformation with parameter $t$ (which she has knowledge of but no choice over). Note that, given how we have defined the rapid game, $t$ will eventually become arbitrarily large, corresponding to \textbf{Bob} choosing a small $\alpha_n$ parameter; otherwise, we get $\inf_n \alpha_n > 0$ and \textbf{Alice} will win by default.
 \par
\textbf{Alice}'s strategy will be to choose $u_x$ such that following the trajectory from the current game state to the next game state, i.e.\ for $0 <  s \leq t$, $g_su_x\Lambda^{\gamma}$ goes far off the cusp (for some maximising $s$) and comes back close to our initial bounded regions at time $t$ itself.
 This means at time $t$ \textbf{Alice} is in a position to apply the default strategy again to get back to $F_\epsilon \cap \pi^{-1}(K_\epsilon)$ and repeat the auxiliary strategy, every time going further down into the cusp and back.
 
 \subsection{The rapid-winning strategy for $\BA^{\gamma} \setminus \BA$}

\subsubsection{Preliminaries}

Before explaining the game strategy, it would be useful to introduce some notation which we will be using extensively. We write $A \lesssim B$ or $A\lesssim_\times B$ when $A\leq cB$ for some constant $c>0$. Also we write $A \asymp B$ when $A\lesssim B \lesssim A$, and we say that the two quantities are comparable. We add a subscript e.g.\ $A \lesssim_{\epsilon} B$ to indicate that the constant $c$ depends on $\epsilon$.\\ 
In the rest of this section, we introduce some lemmas and definitions required to formulate the strategy. The following lemma will help us control the distances of points under a bounded linear transformation.
\begin{lemma}\label{boundedmatrix}
    Let $\xx$ and $\yy$ be points in $\R^2$, and $A \in SL_2(\R)$ a $2\times 2$ matrix transformation with coefficients bounded by some constant $C$. Then if $dist(\xx,\yy) \asymp \epsilon$ for some $\epsilon > 0$, we also have $dist(A\xx,A\yy) \asymp_C \epsilon$.
\end{lemma}
\begin{proof}
We have
\[
\|A\yy - A\xx\| \leq \|A\| \|\yy - \xx\| \lesssim C \epsilon
\]
where $\|A\|$ is the operator norm of $A$. Conversely, since $A\in \SL_2(\R)$ we have $\|A^{-1}\| = \|A\| \lesssim C$ and thus
\[
\epsilon \lesssim \|\yy - \xx\| \leq \|A^{-1}\| \|A\yy - A\xx\| \lesssim C \|A\yy - A\xx\|.
\]
\end{proof}
\begin{definition}
The \textbf{inhomogeneous minimum} $\mu(\Lambda)$ of a lattice $\Lambda \in \R^2$ is the supremum of the distances from all points in $\R^2$ to their closest lattice point. More precisely,
    \[\mu(\Lambda)= \sup_{x \in \R^2}\inf_{\pp \in \Lambda} \| \xx - \pp \|.\]
\end{definition}
The following lemma is due to Jarn\'ik: 
\begin{lemma}
    The inhomogeneous minimum $\mu$ is related to the successive minima
    \[
    \lambda_i = \inf\{\lambda: \Lambda\cap B(0,\lambda) \textup{ contains $i$ linearly independent vectors}\}
    \]
    by the inequality:
    \[\mu \leq \frac{1}{2}(\lambda_1 + \lambda_2 + \dots + \lambda_n).\]
\end{lemma}
\begin{proof}
    ~\cite[Theorem 1 in Chapter 2.13]{GeomNumbGL}.
\end{proof}
Let $\Lambda \in K_\epsilon$, i.e.\ $\lambda_1 \geq \epsilon$. Then by Minkowski's second theorem~\cite[Theorem 1 in Chapter 2.9]{GeomNumbGL}, $\lambda_2 \leq 4 \epsilon^{-1}$. This implies:
\[\mu \leq \frac{1}{2}(\lambda_1 + \lambda_2) \leq \lambda_2 \leq 4 \epsilon^{-1}.\]
\begin{remark}
    The above line of argument states that for a $2$-dimensional bounded lattice in $K_\epsilon$ the inhomogeneous minimum is bounded with the bounding parameter depending on $\epsilon$. This is a simple example of many results connecting successive minima and the inhomogeneous minimum. Further discussion can be found in~\cite{GeomNumbGL}.
\end{remark}

\subsubsection{The Default Strategy}\label{default}

As explained in Section~\ref{sec:3.3}, we are concerned with a game taking place on the space $Y_2$. Our goal is to give a successful strategy for \textbf{Alice}. We prove in Lemmas \ref{lemmaKeps} and \ref{lemmaFeps} below that there is a strategy for \textbf{Alice} for the strong game (and hence by Corollary \ref{corollarySWimpliesRW}, also for the rapid game) which guarantees that
\begin{itemize}
\item The game state always remains in $\pi^{-1}(K_\theta)\cap F_{\theta'}$, where $\theta,\theta' > 0$ depend only on $\delta > 0$, where the game state was initially\footnote{Here ``initially'' means ``from when \textbf{Alice} started applying the default strategy''.} in $\pi^{-1}(K_\delta)\cap F_\delta$.
\item The game state eventually remains in $\pi^{-1}(K_\zeta)\cap F_{\zeta'}$, where $\zeta,\zeta' > 0$ are independent of the initial game state.
\end{itemize}
We call \textbf{Alice}'s strategy to achieve these conditions the \textit{default strategy}.
\begin{lemma}
\label{lemmaKeps}
Assuming $\alpha \leq 1/4$, we can give a strategy for the $(\alpha,\beta)$-strong game with the following property: There exists $\theta>0$ (depending on $\delta$) such that the strategy takes the homogeneous part of the game state from $K_\delta$ to $K_\zeta$ while staying within $K_\theta$ (i.e. how far into the cusp the trajectory needs to go before getting to $K_\zeta$ depends on how far into the cusp we already are). Here $\zeta>0$ is a universal constant independent of $\delta$.
\end{lemma}

Note that since we are dealing with the strong game rather than the rapid game, the continuous trajectory $(g_t u_{x_\infty} \Z^2)_{t\geq 0}$ lies within a bounded distance of the discrete trajectory $(\Lambda_n)_{n\in\N}$ defined by the stages of the game as in Section~\ref{sec:3.3}. So it suffices to show that the discrete trajectory returns to $K_\zeta$.
\begin{proof}
    First, note that since we are working with unimodular lattices in $\R^2$, there is at most one primitive lattice point inside the ball of radius $1$ around the origin. Call this point $(a_1,a_2)$ and without loss of generality assume it lies in the first quadrant. Then we define \textbf{Alice}'s strategy as follows: apply $u_x$ with $x= \alpha -1$ followed by $g_t$ where $t= -\tfrac 12 \log(\alpha)$ resulting in
    
    \begin{equation*} 
\begin{pmatrix}
    1/\alpha & 0\\
    0 & \alpha
\end{pmatrix}
\begin{pmatrix}
    1 & 1-\alpha\\
    0 & 1\\ 
\end{pmatrix}\begin{pmatrix}
    a_1\\
    a_2\\
\end{pmatrix}=\begin{pmatrix}
    \tfrac 1\alpha(a_1+a_2-\alpha a_2)\\
    \alpha a_2\\
\end{pmatrix}. \end{equation*}
Suppose that \textbf{Bob} responds by applying $u_y$ with $|y| \leq 1-e^{-2s}$ followed by $g_s$ where $s \leq -\tfrac 12 \log(\beta)$ resulting in
\begin{equation*} 
\begin{pmatrix}
    e^s & 0\\
    0 & e^{-s}
\end{pmatrix}
\begin{pmatrix}
    1 & -y\\
    0 & 1\\ 
\end{pmatrix}\begin{pmatrix}
    \frac 1\alpha(a_1+a_2 - \alpha a_2)\\
    \alpha a_2
\end{pmatrix}=\begin{pmatrix}
    \tfrac 1{\alpha e^{-2s}}(a_1+a_2-\alpha a_2 - \alpha y a_2)\\
    \alpha e^{-2s} a_2\\
\end{pmatrix}.
\end{equation*}
Now since $\alpha \leq 1/4$ and $e^{-2s} \leq 1$ the distance of this point to the origin is bounded below by
\[
\left|\tfrac1{\alpha e^{-2s}}a_1+\tfrac1{\alpha e^{-2s}}(1-\alpha(1+y))a_2\right| \geq 4a_1 + \tfrac 1{2\alpha e^{-2s}} a_2 > 2(|a_1| + |a_2|)
\]
which means this primitive vector is becoming larger after each round. On the other hand, there are no other primitive vectors of $\Lambda$ of size $<1$, so there are no vectors of $g_s u_y g_t u_x \Lambda$ of size $< \zeta := \tfrac 12\alpha\beta$. Therefore, the smallest primitive vector of $g_s u_y g_t u_x \Lambda$ is at least of size $\min(\tfrac 12\alpha\beta,\Delta(\Lambda))$, where $\Delta(\Lambda)$ is the size of the smallest primitive vector of $\Lambda$. This means that the size of the smallest primitive lattice vector of the game state is increasing until it reaches $\zeta$, at which point it stays above $\zeta$. This implies that eventually the game state will lie in and remain in $K_\zeta$.
\end{proof}

\begin{lemma}
\label{lemmaFeps}
Now assume that the inhomogeneous game state is in $\pi^{-1}(K_\delta) \cap F_\delta$. If $\alpha\leq 1/4$, then there exists a strategy for the $(\alpha,\beta)$-strong game that can take it to $\pi^{-1}(K_\zeta) \cap F_{\zeta'}$ while staying in $\pi^{-1}(K_\theta) \cap F_{\theta'}$ for some $\theta,\theta'>0$ depending on $\delta$. Here $\zeta,\zeta'>0$ are universal constants independent of $\delta$. 
\end{lemma}
\begin{proof}
To start with, let $\alpha' = \alpha$ and $\beta' = \beta\alpha\beta$. Then \textbf{Alice} can alternate between two strategies in the $(\alpha',\beta')$-strong game to get a strategy in the $(\alpha,\beta)$-strong game. For her first strategy, \textbf{Alice} follows the strategy outlined in Lemma \ref{lemmaKeps}, thus guaranteeing that the homogeneous game state stays in $K_\theta$ for some $\theta>0$ depending on $\delta$, and eventually remains in $K_\zeta$ for some $\zeta>0$ independent of $\delta$. Next, for her second strategy in the $(\alpha',\beta')$-strong game, while \textbf{Alice} is in $K_\theta$, consider the ball of radius $\theta/2$ around the origin.  
Since the homogeneous part is in $K_\theta$ we know that this ball can contain at most one grid point, say $\aa\in \Lambda + \rr$. As before, we can choose $u_x$ so as to maximize $\|g_t u_x \aa\|$, and it follows that
\[
\Delta(g_t u_x (\Lambda + \rr)) \geq \min(\tfrac 12\alpha\beta\theta,2\Delta(\Lambda + \rr))
\]
and so as before, we get that the smallest grid vector in the game state is increasing in size until it reaches $\theta' := \tfrac 12\alpha\beta\theta$. It then satisfies $\Delta \geq \theta'$ until the game state is reliably within $\pi^{-1}(K_\zeta)$, at which point we can use the same line of reasoning to conclude that the smallest grid vector in the game state is increasing in size until it reaches $\zeta' := \tfrac 12\alpha\beta\zeta$. It follows that the game state lies in $F_{\theta'}$ until it eventually lies in and remains in $F_{\zeta'}$.
\end{proof}

\subsubsection{The Auxiliary Strategy}\label{auxiliary}
Let \textbf{Alice} play the default strategy until the game state $\Lambda + \rr$ is in $F_{\zeta'} \cap \pi^{-1}(K_\zeta)$. Assume this is the $n$-th stage of the game. This means \textbf{Bob} has given \textbf{Alice} the parameter $s_n$, \textbf{Alice} needs to choose a parameter $z_n$, and the game step will be $g_{s_n}u_{z_n}\Lambda_{B_n}$.  For the sake of clarity,  in order to analyse the $n$-th move of \textbf{Alice}, we  let $x:=z_n$ and $t:=s_n$.
The goal is to find a successful parameter $x(t)$ that provides \textbf{Alice} a winning strategy.  Note that with Section~\ref{sec:3.3} in mind,  every $\aa \in \Lambda$ will be transformed under the following matrix multiplication:
\begin{equation} 
\begin{pmatrix}
    e^\tau & 0\\
    0 & e^{-\tau}
\end{pmatrix}
\begin{pmatrix}
    1 & -x\\
    0 & 1\\ 
\end{pmatrix}\begin{pmatrix}
    a_1\\
    a_2\\
\end{pmatrix}=\begin{pmatrix}
    e^\tau(a_1-xa_2)\\
    e^{-\tau} a_2\\
\end{pmatrix}  \label{trans} \tag{1} \end{equation}
where $0 \leq \tau \leq t$. The same is also true for every translation vector $\rr \in \R^2 \setminus \Lambda$. 
We claim that for successful parameters $x \in \R$ satisfying the following conditions is sufficient.
\begin{enumerate}
   \item[ (\Romannum{1})] At some maximising time $\tau^* \in (0,t)$ the homogeneous part goes into the cusp, i.e. for some $\aa(t) \in \Lambda$, $g_{\tau^*}u_{x(t)}\aa(t)$ tends to $0$ as $t$ tends to $\infty$.
   \item[(\Romannum{2})] At time $\tau=t$, $g_{t}u_{x}\Lambda$ is back in a bounded region i.e. $g_{t}u_{x}\Lambda \in K_\delta$ for some $\delta >0$. 
   \item[(\Romannum{3})] For all time $\tau \in [0,t]$, $g_{\tau}u_{x}(\Lambda+\rr) \in F_\delta$  for some $\delta > 0 $, i.e.\ all representative translation vectors stay bounded away from the origin.
\end{enumerate}
\begin{enumerate}
    \item[(\Romannum{4})] The inequality $\|x\| \leq 1-e^{-t}$ must hold.
\end{enumerate}
The last condition is simply arising from Section~\ref{sec:3.3} in order to satisfy the rules of the game.

After using the auxiliary strategy, \textbf{Alice} uses the default strategy with initial game state $g_t u_x (\Lambda + \rr)$. Conditions (\Romannum{2}) and (\Romannum{3}) together with Lemmas \ref{lemmaKeps} and \ref{lemmaFeps} guarantee that while \textbf{Alice} plays the default strategy, the game state will remain in $\pi^{-1}(K_\theta)\cap F_\theta$ for some $\theta > 0$ depending only on $\delta$. Finally, when the game state goes back into $\pi^{-1}(K_\zeta)\cap F_{\zeta'}$, \textbf{Alice} continues playing the default strategy until \textbf{Bob} plays $\alpha_n \leq 1/k$, where $k$ is the number of times this cycle has happened so far. We know that \textbf{Bob} must eventually play such an $\alpha_n$ since otherwise, he loses by default. When he does, \textbf{Alice} switches to playing the auxiliary strategy and the cycle continues.

Condition (\Romannum{1}) ensures the homogeneous part escapes every bounded region so by Dani's correspondence, the outcome of the game avoids the set $\BA$. Similarly, Condition (\Romannum{3}) together with the fact that in the default strategy the game state remains in $\pi^{-1}(K_\theta)\cap F_{\theta'}$ ensures that the outcome lies in the set $\BA^{\gamma}$, by letting $\epsilon = \min(\theta,\theta',\delta)$ in Theorem \ref{theoremdani}. Thus, if we can show that the auxiliary strategy has the above four properties, then this will show that $\BA^\gamma \setminus \BA$ is rapid winning, thus proving Theorem \ref{maintheoremv2} and hence also Theorem \ref{maintheorem}.

\subsubsection*{Satisfying Condition (\Romannum{1})}
Consider an arbitrary lattice point $\aa \in \Lambda$ and let $x = \frac{a_1}{a_2}+e^{-t}$. Then the transformation will be
\begin{equation*} 
\begin{pmatrix}
    e^\tau & 0\\
    0 & e^{-\tau}
\end{pmatrix}
\begin{pmatrix}
    1 & -\frac{a_1}{a_2}+e^{-t}\\
    0 & 1\\ 
\end{pmatrix}\begin{pmatrix}
    a_1\\
    a_2\\
\end{pmatrix}=\begin{pmatrix}
    e^{\tau-t}a_2\\
    e^{-\tau}a_2\\
\end{pmatrix}. \end{equation*}

Therefore for vector $\aa$ the initial norm is $|a_2|$ which is bounded. Then the maximising point happens exactly at $\tau=t/2$ where $\aa$ has transformed to $e^{-t/2} a_2(1,1)$. This tends to zero as $t$ tends to infinity as required. Finally at time $t$, the norm is back to $|a_2|$ which is necessary but not sufficient for condition (\Romannum{2}).

\begin{remark}
    Note that \textbf{Alice} has yet to choose the particular lattice point $\aa$ which then defines her $u_x$ transformation. Our analysis so far has been for any generic $\aa$.
\end{remark}
\subsubsection*{Satisfying Condition (\Romannum{2})}
Note that while addressing condition (\Romannum{1}) we ensured that the particular lattice point $\aa$ returns to the bounded regions at time $t$. Now we need to prove in fact \textit{every} lattice point $\bb$ stays bounded. This however simply holds in this particular case where we are working with two-dimensional lattices. Let $\aa'$ and $\bb'$ be shorthands for $u_x \aa$ and $u_x \bb$ respectively. Consider the parallelogram spanned by $\aa'$ and $\bb'$. Its area is fixed under the $g_\tau$ transformation and $\aa$ is also shown to stay bounded under the transformation. Therefore the only other spanning vector $\bb$ also has to stay bounded. More precisely:
\[ \|g_t \aa'\|\cdot \|g_t \bb' \| \geq |g_t \aa' \wedge g_t \bb' | = |\aa'\wedge\bb' | \gtrsim_\times 1.\]
Now since $\|g_t \aa'\| = a_2 \asymp_\times 1$, we get $\|g_t \bb' \| \gtrsim_\times 1$.\\
Another case to be considered is when $\aa'$ and $\bb'$ do not span a parallelogram, i.e.\ $\bb'$ lies on the line $\R\aa'$ so $\aa'\wedge\bb' = 0$. This implies $\bb' = n\aa'$ for some $n\in\Z\setminus \{0\}$, and thus $\|g_t \bb' \| \geq \|g_t \aa'\| \asymp_\times 1$.
\subsubsection*{Satisfying Condition (\Romannum{4})}
Here, we address the remark at the end of Condition (\Romannum{1}), Since lattices are symmetric around the origin, let \textbf{Alice} choose from the first quadrant. This means we need to show;
\[
\left| \frac{a_1}{a_2}-e^{-t}\right| \leq 1 - e^{-t}\]
where $e^{-t} \leq \frac{1}{2}$. One can easily see that this holds when $\frac{a_1}{a_2} \leq 1$.
We also claim that $\aa$ can be chosen such that \[\zeta \lesssim \| \aa \| \lesssim \frac{1}{\zeta}.\]
The first inequality is trivial due to the definition of $K_\zeta$. 
Let $\mu$ be the inhomogeneous minimum of the lattice. Then $\mu \leq 4 \zeta^{-1}$. If we draw a circle with radius $\mu$, it will necessarily include a lattice point. It now suffices to draw this circle in the upper half of the first quadrant. For instance, if we draw it tangent to the $y$-axis and the line $\{y=x\}$ we see the furthest we need to go from the origin to hit a lattice point is $\mu(1+\sin(\frac{\pi}{8})^{-1}) \lesssim \frac{1}{\zeta}$.

Thus, there exists a primitive lattice vector $\aa\in\Lambda$ in the upper half of the first quadrant, satisfying $\zeta \leq \|\aa\| \lesssim \zeta^{-1}$. We claim that there is a second primitive lattice vector $\bb\in\Lambda$, such that $\{\aa,\bb\}$ constitutes a spanning set for $\Lambda$, such that $\bb$ is in the upper half of the first quadrant and $\zeta \leq \|\bb\| \lesssim \zeta^{-1}$. Indeed, without loss of generality suppose that $\aa$ lies below the line of angle $3\pi/8$ passing through the origin. Then let $\bb'$ be the intersection of the lines $\{\bb : \aa\wedge\bb = 1\}$ and the $y$-axis, i.e. $\bb' = (0,1/a_1)$. There must be a lattice point $\bb = \bb' + t\aa$ with $0\leq t < 1$. Then
\[
\zeta \leq \|\bb\| \leq \|\bb'\| + \|\aa\| \leq \frac{1}{\sin(\pi/8)\|\aa\|} + \|\aa\| \lesssim \frac{1}{\zeta}\cdot
\]

\subsubsection*{Satisfying Condition (\Romannum{3})}
Let $\aa$ and $\bb$ be two primitive vectors in $\Lambda$ chosen according to Condition (\Romannum{4}). Without loss of generality, let $\rr$ be the grid point inside the parallelogram spanned by $\aa$ and $\bb$ (this is not necessary and only makes the visualisations simpler). Note that $\rr$ is $\zeta$-away from $\aa$ and $\bb$, more generally;
$\dist(\rr , \Lambda) \geq \zeta$.
Now let $A$ be the matrix transformation taking $(0,1)$ to $\aa$ and $(1,0)$ to $\bb$. Let $\pp = A^{-1} \rr$, then by Lemma~\ref{boundedmatrix}, the pre-image in the standard lattice satisfies
\[\dist(\pp , \Z^2) \gtrsim \zeta.\]
This implies one of the coordinates is $\zeta$-away from $\Z$. So without loss of generality, we can assume 
\[\dist(\pp , \Z \times \R) \gtrsim \zeta.\]
Applying $A$ and getting back to $\Lambda$ we get that

\[\dist(\rr,\Lambda + \R\aa) \gtrsim \zeta.\]
\begin{remark}
    The last inequality indeed holds for \textbf{any} grid point $\rr$, not only the particularly chosen representative.
\end{remark} 
 \begin{remark}
      This is assuming \[\dist(\rr,\Lambda + \R\aa) \geq \dist(\rr,\Lambda + \R\bb),\] and one can swap the choice of $\aa$ and $\bb$ if needed.
 \end{remark}
 The last remark completes \textbf{Alice}'s strategy. \textbf{Alice} needs to pick a lattice point $\aa$ such that $\Lambda + \R\aa$ is $\zeta$-away from $\rr$. We claim this fulfils Condition (\Romannum{3}).
Let $\xx' \in \Lambda' + \rr'$ with the prime symbol denoting the result after \textbf{Alice}'s move as before. We need to show that $\|g_s \xx'\| \gtrsim_\times 1$ for $0 \leq s \leq t$. The idea is to use the fact that $\|g_s \aa'\|$ is bounded:

\[\|g_s \aa'\|\cdot \|g_s \xx'\| \geq |g_s \aa' \wedge g_s \xx'| = |\aa'\wedge\xx'| \gtrsim_\times 1.\]
In the last inequality, we are using the fact that since $\rr$ is away from $\Lambda + \R\aa$, $\xx'$ is also away from $\Lambda + \R\aa'.$ Now since $\|g_s \aa'\| = a_2\max(e^{-s},e^{-(t -  s)}) \lesssim_\times 1$, we get $\|g_s \xx'\| \gtrsim_\times 1$.

\newpage

\bibliographystyle{amsplain}

\bibliography{main}
\end{document}